\documentclass[12pt,onecolumn, draftcls]{IEEEtran}
\usepackage{cite}
\usepackage{amsmath,graphicx}
\usepackage{amsthm,amssymb}
\usepackage{epstopdf,dsfont,float,cite,color,array,bm}
\usepackage{lscape,url}
\usepackage{mathtools}
\usepackage{caption,subcaption}

\usepackage{graphicx}
\usepackage{tikz}
\usepackage{soul, comment}

\newtheorem{Theorem}{Theorem}

\newtheorem{lemma}{Lemma}
\newtheorem{Corollary}{Corollary}

\theoremstyle{definition}
\newtheorem{Remark}{Remark}
\date{}

\def\thetanorm{{\lVert \btheta \rVert^2}}

\def\IHAT{{\mathsf I(\bthetahat)}}

\def\bmu{{\boldsymbol{\mu}}}
\def\bI{{\mathbf{I}}}
\def\bZ{{\mathbf{Z}}}

\def\bX{{\mathbf{X}}}

\def\bbeta{{\bm{\beta}}}

\def\thetaML{{\hat{\btheta}^{\textsc{ML}}}}
\def\thetaJS{{\hat{\btheta}^{\textsc{JS}}}}
\def\thetaJSm{{\hat{\btheta}^{\textsc{JS}}_m}}
\def\thetaJSmplus{{\hat{\btheta}^{\textsc{JS}}_{m+}}}

\def\thetaJSV{{\bthetahat^{\textsc{JS}}_\mathbb{V}}}

\def\thetaJSplus{{\hat{\btheta}^{\textsc{JS}}_+}}

\def\barX{{\overbar{\bX}}}

\def\calN{{\cal N}}
\def\calF{{\cal F}}
\def\calC{{\cal C}}
\def\calD{{\cal D}}

\def\Ex{{\mathbb{E}}}
\def\Extheta{\Ex_1^\btheta}

\def\P{{\mathbb P}}

\def\bI{{\mathbf{I}}}
\def\btheta{{\boldsymbol{\theta} }}
\def\bthetahat{{\hat{\boldsymbol{\theta} }}}

\def\bbR{{\mathbb{R}}}

\def\bbV{{\mathbb{V}}}
\def\projV{{P_\mathbb{V}}}

\def\KL{{\mathsf I(\btheta)}}

\def\bX{{\mathbf{X}}}
\def\bZ{{\mathbf{Z}}}

\def\TGLR{{T^{\textsc{G}}_b}}

\def\TSRRS{{T^{\textsc{S}}_b}}

\def\TWL{{T^\textsc{WL}_b }}

\newcommand\given[1][]{\:#1\vert\:}
\newcommand\norm[1]{\lVert#1\rVert}
\newcommand\delay[1]{\calD^\btheta\left(#1\right)}
\newcommand\MSE[1]{\mathsf{MSE}_\btheta(#1)}
\def\define{\coloneqq}

\newcommand{\overbar}[1]{\mkern 1.5mu\overline{\mkern-1.5mu#1\mkern-1.5mu}\mkern 1.5mu}

\def\blue#1{{\color{black}#1}}

\begin{document}
\bstctlcite{MyBSTcontrol}
\title{Quickest Change Detection for Multiple Data Streams Using the~James-Stein~Estimator}

\author{Topi~Halme,~\IEEEmembership{Student Member,~IEEE},~Venugopal~V.~Veeravalli,~\IEEEmembership{Fellow,~IEEE}
and ~Visa~Koivunen,~\IEEEmembership{Fellow,~IEEE}%

\thanks{T. Halme and V. Koivunen are with the Department of Information and Communcations Engineering, Aalto University, Espoo, Finland. (e-mail: \{topi.halme, visa.koivunen\}@aalto.fi). V.V. Veeravalli is with the ECE Department and Coordinated Science Laboratory, The Grainger College of Engineering, University of Illinois at Urbana-Champaign, Urbana, IL, 61820 USA. (e-mail: vvv@illinois.edu). V.V. Veeravalli was supported in part by a Fulbright-Nokia Distinguished Chair in Information and Communication Technologies during his visit to Aalto University in Fall 2023 when this work was initiated.}
}
\IEEEoverridecommandlockouts

\maketitle
\vspace*{-0.5in}
\begin{abstract}
The problem of quickest change detection is studied in the context of detecting an arbitrary unknown mean-shift in multiple independent Gaussian data streams. The James-Stein estimator is used in constructing detection schemes that exhibit strong detection performance both asymptotically and non-asymptotically. Our results indicate that utilizing the James-Stein estimator in the recently developed window-limited CuSum test constitutes a uniform improvement over its typical maximum likelihood variant. That is, the proposed James-Stein version achieves a smaller detection delay simultaneously for all possible post-change parameter values and every false alarm rate constraint, as long as the number of parallel data streams is greater than three. Additionally, an alternative detection procedure that utilizes the James-Stein estimator is shown to have asymptotic detection delay properties that compare favorably to existing tests. The second-order asymptotic detection delay term is reduced in a predefined low-dimensional subspace of the parameter space, while second-order asymptotic minimaxity is preserved. The results are verified in simulations, where the proposed schemes are shown to achieve smaller detection delays compared to existing alternatives, especially when the number of data streams
is large.
\end{abstract}

\begin{IEEEkeywords}
Sequential change detection, James-Stein estimator, shrinkage estimation, average detection delay
\end{IEEEkeywords}

\section{Introduction}

The detection of changes in the underlying statistical properties of online data streams is an important task in a variety of domains ranging from sensor networks, wireless communications, radar and the Internet of Things (IoT) to biosurveillance and medical image processing. 
In the quickest change detection (QCD) problem, a sequence of observations undergoes a change in its distribution at an unknown time. The objective of the observer is to detect the change with minimal delay, subject to constraints on the prevalence of false alarms. For general references on quickest change detection, we refer to \cite{TARTAKOVSKY_BOOK, POOR_BOOK, XIE_2021_REVIEW, VEERAVALLI_2014}.

In this paper, we focus on a setting where the change simultaneously affects a large number of data streams with unknown, possibly differing, magnitudes under Gaussian noise. Such a formulation arises naturally when, e.g., a distributed sensor network with $K$ sensors is used to monitor events that vary in time and space. The intensity of the event or disruption observed by a particular sensor may be related to its spatial proximity or its overall sensing capability, which can vary from sensor to sensor. Mathematically, this model corresponds to detecting an arbitrary shift in the mean vector $\btheta$ of a $K$-dimensional multivariate Gaussian distribution. 

In change detection with unknown post-change parameters, common procedures can generally be divided into three main categories: generalized likelihood ratio (GLR) based \cite{LORDEN_1971, LAI_1998}, mixture approaches \cite[sec. 8.3.]{TARTAKOVSKY_BOOK}, and adaptive estimates based \cite{SPARKS_2000,LORDEN_2005, CAO_2018, XIE_2023}. Particularly common is a GLR extension of the CuSum procedure, originally introduced in \cite{LORDEN_1971}. The GLR-CuSum and its computationally simpler windowed version \cite{LAI_1998} are asymptotically first-order optimum  under very general conditions (including non-Gaussian models) as the average run length (ARL) to false alarm $\gamma$ tends to infinity. Recently, a window-limited adaptive CuSum (WL-CuSum) test \cite{XIE_2023} was developed as an alternative to the GLR-CuSum test. The WL-CuSum test possesses similar first-order optimality properties as the windowed GLR-CuSum test, while being less demanding computationally \cite{XIE_2023}.  

While the GLR and WL-CuSum tests are first-order asymptotically optimal for any fixed dimension $K$, the higher-order terms of the asymptotic average detection delay are heavily influenced by $K$. In exponential families, the second-order term of the asymptotic average delay (on the order of $\log \log \gamma$) of GLR-CuSum increases linearly in $K$ \cite[pp. 428]{TARTAKOVSKY_BOOK}. In this paper, we show that the effect of the dimension on the detection delay can be reduced by estimating the unknown parameter with shrinkage estimators \cite{CANDES_2006, FOURDRINIER_BOOK}. In particular, we propose using the James-Stein estimator \cite{JAMES_STEIN_1961} in place of the standard maximum likelihood estimator in existing adaptive change detection tests. The James-Stein estimator dominates the maximum likelihood estimator in terms of mean squared-error (MSE) when estimating the mean of a multivariate Gaussian distribution when $K \geq 3$. This paper shows that this property can be exploited to significantly improve adaptive change detection tests such as those presented in \cite{LORDEN_2005} and \cite{XIE_2023} in Gaussian data models.

Although the James-Stein estimator has been widely recognized and studied in the mathematical statistics community, there are only very few direct applications of it in statistical signal processing and related fields. 
It has been utilized for tracking the state of noisy dynamical systems \cite{MANTON_1998, CAMPANA_2018}, for improved estimation of the entropy of categorical random variables \cite{HAUSSER_2009}, and for image denoising \cite{WU_2013_IMAGE, NGUYEN_2017_IMAGE}. We are not aware of any previous applications of the James-Stein estimator to sequential change detection.

In the context of change detection, general shrinkage estimation has been previously considered in \cite{WANG_2015}. It was shown that linear and hard-thresholding shrinkage estimators can be used to reduce the second-order detection delay term in Gaussian data models. This may lead to a smaller delay for moderate ARL when $K$ is sufficiently large. However, the gains obtained by the shrinkage estimators considered in \cite{WANG_2015} come at the cost of sacrificing first-order asymptotic optimality.
In contrast, we show that utilizing the James-Stein estimator can reduce the second-order delay term in a predefined low-dimensional subspace of the post-change parameter space while simultaneously maintaining second-order asymptotic minimaxity, a condition that is stronger than first-order optimality.

There exists a considerable body of literature on studying the case where the change affects only a portion of the data streams, i.e. $\btheta$ has a small number of non-zero components. 
The special case where the change is observed by only one sensor was studied in \cite{TARTAKOVSKY_2006, TARTAKOVSKY_2004}, see also \cite[Ch. 9]{TARTAKOVSKY_BOOK}. Methods for the setting where the change affects an unknown arbitrary set of streams were proposed and analyzed in \cite{MEI_2010, FELLOURIS_2016} under the assumption that the pre- and post-change distributions are completely specified, and in \cite{XIE_2013, LIU_2019} under post-change parametric uncertainty. The test proposed in \cite{XIE_2013} is a mixture procedure in which each stream is hypothesized to be affected with a pre-specified probability $p$. If $p = 1$, implying that all streams are expected to be affected, the procedure reduces to the GLR test. The procedures we propose in this paper can also be used for detecting such sparse changes. The procedures are not, however, designed specifically with the detection of sparse changes as the objective. In fact, the James-Stein estimator is known to be most effective when the estimated vector is not too sparse \cite{CANDES_2006}, and therefore we expect our proposed tests to be most useful when the proportion of affected data streams is reasonably large.

The main contributions of the paper are summarized below:
\begin{enumerate}
\item We propose using the James-Stein shrinkage estimator for quickest change detection in multiple data streams.
\item \blue{We show that deploying the (positive-part) James-Stein estimator in place of the maximum likelihood estimator in the WL-CuSum test can significantly reduce the expected detection delay in the detection of a Gaussian mean-shift. We derive a non-asymptotic upper bound on the detection delay the WL-CuSum test which is valid for any ARL constraint~$\gamma$, post-change parameter $\btheta$, and estimator $\bthetahat$ subject to some technical assumptions. The bound is seen to crucially depend on the MSE of the chosen estimator. Since the JS estimator dominates ML estimator in MSE, our results indicate that the JS estimator consitutes a uniform improvement over the ML estimator in the considered setting.} The JS estimator is easy to compute from the sufficient statistic (which in the considered setting is equal to the MLE), and hence extra computation required for the James-Stein based test is negligible.
\item If computational efficiency is not paramount, we show that extending the SRRS test of \cite{LORDEN_2005} with the James-Stein estimator yields a test that is second-order asymptotically minimax and superior compared to existing procedures such as the GLR-CuSum. The second-order detection delay term of the proposed JS-SRRS test is shown to be independent of $K$ in a prespecified lower dimensional subspace of the parameter space. For parameter values outside this subspace, the first and second-order terms of the detection delay match those of common alternatives. Additionally, we highlight a connection between the improvement of the second-order term and the superefficiency of the James-Stein estimator.
\item Overall, the analytical results and the simulation experiments verifying them suggest that in the considered setting the proposed JS-SRRS and JS-WL-CuSum are uniform improvements over their maximum likelihood counterparts. They also perform favorably in comparison to GLR-CuSum when $K$ is moderate to large in simulations.
\end{enumerate}

The rest of this paper is structured as follows. In Section \ref{sec:prelim}, we formulate the problem, and provide background on related existing change detection procedures. The James-Stein estimator and its properties are also introduced. In Section \ref{sec: main}, we propose and analyze the JS-WL-CuSum and JS-SRRS change detection tests. In Section \ref{sec:sim}, we present simulation experiments that demonstrate the utility of the proposed tests. In Section \ref{sec: conclusion}, we provide some concluding remarks.
\smallskip

\noindent\textbf{Notation.} 
For time steps $m$ and $n$, $m < n$, we use $\bX_m^n$ to 
denote the set of variables $\{\bX_m,\ldots,\bX_n\}$. The value of an estimator $\bthetahat$ applied to $\bX_m^{n-1}$ is denoted by $\bthetahat_{m,n}$, meaning that $\bthetahat_{m,n}\define\bthetahat(\bX_m^{n-1})$. 
\section{Preliminaries}\label{sec:prelim}
\subsection{Problem formulation}

Let $\{\bX_n\}$ be a sequence of sequentially observed multivariate Gaussian random variables with dimension $K \geq 3$. Suppose there exists a deterministic and unknown change-point $\nu$ such that
\begin{equation}
\bX_n \overset{\mbox{i.i.d}}{\sim} \begin{cases}
~~\calN(\btheta_0, \sigma^2\bI), \quad &n = 1,2,...,\nu -1  \\
~~\mathcal{N}(\btheta, \sigma^2\bI), ~~~ &n = \nu, \nu+1,... 
\end{cases}
\end{equation}
i.e., a mean shift from $\btheta_0$ to $\btheta$ occurs at time $\nu$ for some $\btheta \in \bm\Theta \subseteq \mathbb R^K \setminus \{\bm \btheta_0\}$. Assume that $\btheta_0$ and $\sigma^2$ are known, or can e.g. be estimated from pre-existing training data sufficiently well to justify the assumption. We can then set $\btheta_0 = \bm 0$ and $\sigma^2=1$ without loss of generality by considering the normalized variables $(\bX_n - \btheta_0)/\sigma.$ The setup can be thought of as observing $K$ independent data streams that undergo a change in mean at a common time $\nu$. In general, it is only assumed that $\btheta \not = \bm 0$, i.e. the change can affect an arbitrary subset of the streams, and the magnitude of the change may be different for each stream. For practical use, it may be reasonable to restrict the set $\bm \Theta$ to include only practically meaningful changes, by e.g. setting $\bm \Theta = \{\btheta : \norm\btheta > \vartheta\}$ for some barrier value $\vartheta \geq 0$. 

We denote the filtration generated by the observation sequence by $\{\calF_n\}$, i.e. $\calF_n = \sigma(\bX_1,...,\bX_n)$. The probability measure when the change to $\btheta$ occurs at time $\nu$ is denoted by $\P_\nu^\btheta$, and $\Ex_\nu^\btheta$ denotes the corresponding expectation. If a change never occurs, the probability measure and expectation are denoted by $\P_\infty$ and $\Ex_\infty$, respectively. The general goal is to detect the change as quickly as possible while controlling the rate of false alarms. Mathematically, a sequential change detection procedure is a stopping time $T$ adapted to $\{\calF_n\}$, meaning that $\{T=n\} \in \mathcal{F}_n.$

\noindent\textbf{False Alarm Criterion.} The rate of false alarms for a procedure $T$ is quantified by the average run length (ARL) to false alarm in the pre-change regime. The family of tests with ARL at least $\gamma$ is denoted by $\calC_\gamma$, i.e. 
 \begin{equation}
     \calC_\gamma \define \left\{T : \Ex_\infty(T) \geq \gamma\right\}.
 \end{equation}

\noindent\textbf{Delay Criterion.}
As is standard, the detection delay is measured in the "worst worst-case" as defined by Lorden  \cite{LORDEN_1971}. For any $\btheta \in \bm \Theta$ 
\begin{equation}\label{lorden}
\delay T \define \underset{\nu\geq 1}{\sup}\text{ ess} \sup \Ex^{\btheta}_\nu((T-\nu+1)^+ | \mathcal{F}_{\nu-1}),
\end{equation}
with the worst-case being considered over both all possible change-points $\nu$ and pre-change observations.

We write $f(\cdot, \btheta)$ for the density function of a $\calN(\btheta, \bI)$ random variable. For any $\btheta \in \bm \Theta$,  the log-likelihood ratio between the hypotheses $H_1^{\btheta}: \{\nu = t,\btheta = \btheta\}$ and $H_0: \{\nu > n\}$ at time $n$ is denoted by $\Lambda^{\btheta}_{t,n}$, i.e.
\begin{equation}\label{eq: LLR_lambda_def}
    \Lambda^{\btheta}_{t,n} \define \sum_{m = t}^n \log \frac{f(\bX_m, \btheta)}{f(\bX_m, \bm 0)} = \sum_{m = t}^n\left( {\btheta}^\top \bX_m - \frac{1}{2} \norm{\btheta} ^2\right).
\end{equation}
Finally, the Kullback-Leibler divergence between $f(\cdot, \btheta)$ and $f(\cdot, \bm 0)$ is denoted by $\KL$, which for $\sigma^2 =1$ is given by
\begin{equation}\label{eq: KL_def}
    \KL \define \Ex^\btheta_1\left[\Lambda^\btheta_{1,1}\right] = \frac{\norm\btheta^2}{2}.
\end{equation}

\subsection{Review of existing procedures}\label{sec: review}

If the post-change parameter vector $\btheta$ is known, it is well known \cite{MOUSTAKIDES_1986} that the following CuSum test $T^\textsc{C}_b$ 
\begin{equation}\label{eq: CuSum definition}
    T^\textsc{C}_b \define  \inf\left\{n : \max_{1\leq t \leq n} \Lambda^\btheta_{t,n} > b\right\},
\end{equation}
exactly minimizes $\calD_\btheta(T)$ in the class $T \in \calC_\gamma$ if the stopping threshold $b$ is chosen such that $\Ex_\infty(T^\textsc{C}_b) = \gamma$. The CuSum statistic $W_n \define \max_{1\leq t \leq n} \Lambda^\btheta_{t,n}$ in \eqref{eq: CuSum definition} can be updated via the recursion
\begin{equation}
    W_n = W_{n-1}^+ + \log \frac{f(\bX_n, \btheta)}{f(\bX_n, \bm 0)} 
\end{equation}
where $z^+ = \max(0, z)$ and $W_0 = 0$.
As $\gamma \to \infty$, the CuSum test with $b = \log \gamma$ satisfies $T^\textsc{C}_b \in \calC_\gamma$ and \cite[Lemma 1]{XU_2021} %
\begin{equation}\label{eq: CuSum delay}
    \delay{T^\textsc{C}_b} = \frac{\log\gamma}{\KL} + O(1),
\end{equation}
where $O(1)$ denotes a constant that does not depend on $\gamma$. We note that the expression in \eqref{eq: CuSum delay}, as well as other results highlighted in this subsection are not just for the Gaussian mean-shift model, but apply to more general probability models as well. In the considered Gaussian case it is possible to find accurate approximations for the constant terms for some tests (see e.g. \cite{TARTAKOVSKY_BOOK, SIEGMUND_1995}), though we do not consider them in this paper.

When the post-change parameters are not known, a natural approach that is popular in the literature is a generalized likelihood ratio based double maximization procedure \cite{LAI_1998, SIEGMUND_1995} 
\begin{align}
&\TGLR \define \inf\left\{n : \max_{1\leq t \leq n}\sup_{\btheta \in \bm \Theta} \Lambda^{\btheta}_{t,n} > b\right\} \label{eq:GLR_stat}\\
&= \inf\left\{n : \max_{1\leq t \leq n}\sum_{m = t}^n\left( {(\overbar{\bX}_{t}^n)}^\top \bX_m - \frac{1}{2} \norm{\overbar{\bX}_{t}^n} ^2\right) > b\right\}\label{eq:GLR_stat_Gaussian}
\end{align}
where $\overbar{\bX}_{t}^n \in \mathbb R^K$ is the sample mean of $\{\bX_t,\ldots, \bX_n\}$. Unlike the CuSum test, the test statistic in \eqref{eq:GLR_stat_Gaussian} cannot be updated recursively.
As $\gamma \to \infty$, the detection delay of the GLR test with a properly chosen threshold grows as \cite[sec. 8.3.]{TARTAKOVSKY_BOOK}
\begin{equation}\label{eq:GLR_delay}
    \delay\TGLR = \frac{\log\gamma}{\KL} + \frac{K\log\log\gamma}{2\KL} + O(1),
\end{equation}
for all $\btheta \in \bm \Theta$. Since the first-order term in \eqref{eq:GLR_delay} matches the  first-order term of the optimal CuSum delay \eqref{eq: CuSum delay}, the GLR test is said to be asymptotically first-order optimal for all $\btheta$. However, there is a price paid for the parametric uncertainty that manifests as the additional $\log \log \gamma$ term. In fact, an information-theoretic lower bound from \cite{SIEGMUND_2008} shows that the factor $K/(2\KL)$ is unavoidable for the second-order term in the following minimax sense:
\begin{lemma}[Information lower bound]\label{lemma: inf_LB}
Let $\bm \Theta$ satisfy sufficient continuity conditions as detailed in \cite[Thm. 1]{SIEGMUND_2008}. Then, as $\gamma \to \infty$
\begin{equation}\label{eq: inf_LB}
    \inf_{T \in \calC_\gamma}\sup_{\btheta \in \bm \Theta} \KL \delay T \geq \log \gamma + \frac{K\log\log\gamma}{2} + O(1).
\end{equation}
\end{lemma}
\noindent We say that a change-detection test is second-order asymptotically minimax \cite{TARTAKOVSKY_BOOK} if it achieves the lower bound \eqref{eq: inf_LB} up to the $O(1)$ term.

A consequence of the double maximization in \eqref{eq:GLR_stat} is that the test statistic of the GLR test is not a proper likelihood ratio (crucially $\Ex_\infty(\sup_\btheta f(\bX_n, \btheta)/f_0(\bX_n)) \not = 1$). 
An alternative approach, which preserves the likelihood ratio property, is to evaluate the likelihood ratio of the $n$th sample with respect to an estimator $\bthetahat$ of $\btheta$ which does not depend on $\bX_n$. This idea was first proposed in the context of sequential hypothesis testing by Robbins and Siegmund \cite{ROBBINS_1972, ROBBINS_1974}, and adopted to change detection by Lorden and Pollak \cite{LORDEN_2005}. Concretely, the Shiryaev-Roberts-Robbins-Siegmund (SRRS) test $\TSRRS$ is given by
\begin{align}
    \widehat\Lambda_{t,n} &\define \sum_{m = t}^{n} \log \frac{f(\bX_m, \bthetahat_{t,m})}{f(\bX_m)}  \\
    &= \sum_{m = t}^n\left( {(\bthetahat_{t,m})}^\top \bX_m - \frac{1}{2} \norm{\bthetahat_{t,m}} ^2\right) \label{eq: SRRS_stat}\\
    \TSRRS &= \inf\left\{n : \sum_{t = 1}^n e^{\widehat\Lambda_{t, n}} > e^b\right\}.
\end{align}
where $\bthetahat_{t,m} = \bthetahat(\bX_t,\ldots,\bX_{m-1})$ is an estimate of $\btheta$ based on $\{\bX_t,\ldots,\bX_{m-1}\}$, and $\bthetahat_{t,t} = \bm 0$. As the estimator $\bthetahat$, Lorden and Pollak \cite{LORDEN_2005} consider method of moments and maximum likelihood estimators in exponential families, and show that the asymptotic performance of the SRRS test depends on the asymptotic efficiency of the used estimator. If $\bthetahat$ is the ML-estimator, i.e. $\bthetahat_{t,m} = \overbar{\bX}_{t,m-1}$, then as $\gamma \to \infty$ \cite{ROBBINS_1974, LORDEN_2005}
\begin{equation}
    \delay\TSRRS = \frac{\log\gamma}{\KL} + \frac{K\log\log\gamma}{2\KL} + O(1),
\end{equation}
which differs from the asymptotic ADD of the GLR test in \eqref{eq:GLR_delay} by at most a constant for all $\btheta.$ Therefore, the SRRS test with the ML-estimator is also second-order asymptotically minimax. In \cite{WANG_2015} it is shown that for large $K$ it is possible to improve the SRRS test for moderate ARL levels by replacing the MLE with a linear shrinkage estimator of the form $\bthetahat_{t,m} = a\overbar\bX_{t}^ {m-1}$, where $0<a<1$. However, the test obtained via linear shrinkage is not first-order asymptotically optimal, and the values of $a$ which yield an improvement depend on the unknown $\btheta$, which might make choosing $a$ difficult in practical use.

As in the GLR test, the number of computations required for updating the SRRS test statistic increases with $n$, which leads to a high computational burden if the test is run for a long time. As a solution, \cite{XIE_2023} suggests using a sliding window of size $w$ for parameter estimation, and employing a CuSum-like test given by
\begin{align}\label{eq:WL_stat} S_n &\define S_{n-1}^+ + \log \frac{f(\bX_n, \bthetahat_{n-w,n})}{f(\bX_n, \bm 0)} \\
\TWL &\define \inf\{n > w : S_n > b\}\label{eq:WL_stop}
\end{align} 
with $S_0 = 0$. Due to the simple recursive form for  $S_n$, the WL-CuSum test is easy to compute, especially if the estimator sequence $\{\bthetahat_{1,w+1}, \bthetahat_{2,w+2},\ldots\}$ also admits a recursion. The theory developed in \cite{XIE_2023} is generally agnostic to any specific estimator, only requiring that the estimator be consistent. It is shown that for large $\gamma$, if $w \to \infty$ such that $w = o(\log \gamma)$
\begin{equation}
    \delay\TWL \leq \frac{\log \gamma}{\KL} + \Theta\left(\sqrt{\log \gamma}\right),
\end{equation}
implying that the WL-CuSum test is asymptotically first-order optimal. However, the tradeoff for the reduced computation is a worse asymptotic second-order term in comparison to the SRRS and GLR procedures. 
\subsection{The James-Stein Estimator}\label{subsec:JS}

Let $\bX \sim \mathcal{N}(\btheta,  w^{-1}\bI)$, with $w$ known and $\btheta \in \mathbb R^K$ unknown. For example, if $\bX_1,...,\bX_{w}$ are independent samples from $\calN(\btheta, \bI)$, the sample mean $\overbar{\bX}_{1}^w$ is a sufficient statistic for estimating $\btheta$ and $\overbar{\bX}_{1}^w \sim \calN(\btheta,  w^{-1}\bI)$. A remarkable result due to James and Stein \cite{JAMES_STEIN_1961} is that when $K \geq 3$, the maximum likelihood estimator $\bthetahat^\textsc{ML}(\bX) = \bX$ is inadmissible for the estimation of $\btheta$ under mean squared-error loss. That is, there exists another estimator $\bthetahat$ such that $\MSE{\bthetahat} \leq \MSE\thetaML$ for all $\btheta$ where
\begin{equation}
    \MSE{\bthetahat} \define \Ex^\btheta\norm{\btheta - \bthetahat}^2,
\end{equation}
with the inequality being strict for at least one $\btheta$. James and Stein showed that the estimator
\begin{equation}\label{eq: JS_definition}
    \hat{\btheta}^{\text{JS}}(\bX) \define \left(1 - \frac{(K-2)}{w\lVert \bX-\bmu \rVert^2}\right)(\bX - \bmu) + \bmu 
\end{equation} 
for any $\bmu \in \mathbb{R}^K$ achieves strictly smaller MSE for all $\btheta$ than the maximum likelihood estimator, see Fig. \ref{fig:mseplot}. The mean squared-error of the James-Stein estimator is equal to (e.g. \cite[pp. 274]{LEHMANN_BOOK})
\begin{equation}\label{eq: JS_MSE}
    \MSE\thetaJS = K/w - \frac{(K-2)^2}{w^2}\Ex^\btheta\left(\frac{1}{\norm{\bX-\bmu}^2}\right),
\end{equation}
where the first term equals the MSE of the ML-estimator and the second term is always negative.
\begin{figure}
    \centering
    \includegraphics[]{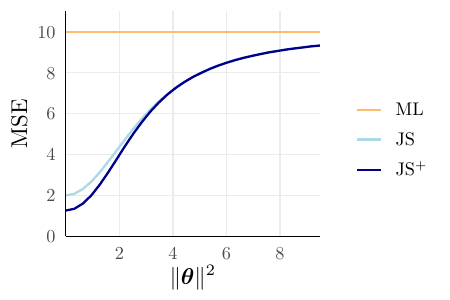}
    \caption{Mean squared-errors (MSE) of the maximum likelihood, James-Stein, and positive-part James-Stein estimators (with shrinkage towards $\bm 0$) when $K = 10$.}
    \label{fig:mseplot}
\end{figure}
Intuitively, the JS estimator shrinks each component of $\bX$ towards $\bmu$ by a factor that depends on $\norm{\bX-\bmu}$. As seen from \eqref{eq: JS_MSE}, the MSE of the JS estimator is the smallest when $\bmu$ is close to $\btheta$. While the JS estimator dominates the maximum likelihood estimator, it is itself inadmissible, as it can be uniformly improved by constraining the shrinkage factor to be non-negative. This modification results in the positive-part James-Stein estimator $\thetaJSplus$ \cite[pp. 357]{LEHMANN_BOOK}
\begin{equation}\label{eq:positive_JS}
\thetaJSplus(\bX) \define \left(1 - \frac{(K-2)}{w\lVert \bX-\bmu \rVert^2}\right)^+(\bX - \bmu) + \bmu,
\end{equation}
for which $\MSE\thetaJSplus < \MSE\thetaJS$ for all $\btheta$. It is noteworthy that even $\thetaJSplus$ is inadmissible, as it is not smooth enough to be an admissible estimator. However, as described in e.g. \cite{EFRON_MORRIS_1973}, substantial improvements over $\thetaJSplus$ do not seem to exist.

\begin{Remark}
The James-Stein estimator achieves the largest reduction in MSE when the chosen shrinkage direction $\bmu$ is close to the true parameter $\btheta$. If the two coincide, $\mathsf{MSE}_\bmu(\thetaJS) = 2/w$, i.e. the MSE is completely independent of the dimension $K$.
\end{Remark}
\begin{Remark}\label{remark:global_mean}
The shrinkage direction can be chosen adaptively based on the data. For example, if there is reason to expect that the components of $\btheta$ are likely to have similar values, a reasonable shrinkage direction might be the global mean $\mathbf m = (1/K)\bm{1}^\top\bX$. Then, the estimator 
\begin{equation}\label{eq:global_mean JS}
    \thetaJSm(\bX) = \left(1 - \frac{(K-3)}{w\lVert {\bX}-
    \mathbf{m}\rVert^2}\right)({\bX}-
    \mathbf{m}) + \mathbf{m}
\end{equation}
dominates the MLE if $K\geq 4$. Moreover, if $\btheta = \theta\mathbf{1}$ for any $\theta \in \mathbb R$, then $\mathbb{E}^\btheta\lVert \btheta - \thetaJSm\rVert^2 = 3/w$ \cite{LEHMANN_BOOK}. 
More generally, the shrinkage can be made toward any linear subspace $\mathbb{V} \subset \mathbb{R}^K$ \cite{SRINATH_2018, GEORGE_1986, FOURDRINIER_BOOK}. If $\projV(\bX)$ denotes the projection of $\bX$ onto $\mathbb{V}$, a James-Stein estimator shrinking toward $\mathbb{V}$ is defined as \cite[pp. 51]{FOURDRINIER_BOOK}
\begin{equation}
\begin{split}\label{eq: general_JS}
    &\bthetahat^{\textsc{JS}}_\mathbb{V}(\bX) \\
    &\define \left(1 - \frac{(K-d-2)}{w\lVert {\bX}-
    \projV(\bX)\rVert^2}\right)({\bX}-
    \projV(\bX)) + \projV(\bX),
\end{split}
\end{equation}
where $d$ is the dimension of $\mathbb{V}$. For instance, if one expects that $\btheta \approx \bm Z\bm\beta$ for some known full-rank matrix $\bm Z \in \mathbb{R}^{K\times d}$ and unknown $\bm\beta \in \mathbb{R}^d$, then $\mathbb{V} = \{\btheta : \btheta = \bm Z\bm\beta\}$ and $\projV(\bX) = \bZ(\bZ^\top\bZ)^{-1}\bZ^\top\bX.$ The estimator in \eqref{eq: general_JS} dominates the MLE as long as $d < K-2$, and its MSE is given by \cite[pp. 52]{FOURDRINIER_BOOK}
\begin{equation}\label{eq: JSV_mse}
    \MSE{\bthetahat^\textsc{JS}_\mathbb{V}} = K/w - (K-d-2)^2\Ex^\btheta\left[\frac{1}{\norm{\bX-\projV(\bX)}^2}\right].
\end{equation}
In the case that $\btheta \in \mathbb{V}$, the MSE reduces to \cite{GEORGE_1986, FOURDRINIER_BOOK}
\begin{equation}\label{eq: JSV_mse_small}
    \MSE{\bthetahat^\textsc{JS}_\mathbb{V}} = \frac{d+2}{w}.
\end{equation}
As with the regular James-Stein estimator, taking the positive part of the shrinkage factor in \eqref{eq: general_JS} yields an estimator with uniformly smaller MSE. The global mean James-Stein estimator in \eqref{eq:global_mean JS} is seen to be a special case of \eqref{eq: general_JS} if $\mathbb{V}$ is the subspace spanned by $\bm 1,$ the all-ones vector. Similarly, $\thetaJSV$ reduces to the ordinary James-Stein estimator \eqref{eq: JS_definition} if $\bbV$ consists of a single point $\bmu$. By choosing $\bbV$ it is possible to flexibly encode existing prior information related to $\btheta$ into choosing an effective shrinkage direction. Importantly, integrating such a prior hypothesis using James-Stein estimation can be done with no penalty in terms of MSE. Even if the hypothesis is entirely incorrect, the MSE is smaller than that attained through maximum likelihood estimation.

\end{Remark}

\section{Change detection using the James-Stein estimator}\label{sec: main}
On a high level, for all change detection procedures highlighted in Section \ref{sec: review}, the test statistic should remain small in the pre-change regime to maintain a satisfactory ARL, and grow quickly after the change for swift detection. With the SRRS and WL-CuSum tests, which require specifying an estimator $\bthetahat$, the desired pre-change behavior can be attained with any estimator, while the post-change behavior heavily depends on the chosen estimator. Consider, for example, the WL-CuSum statistic $S_n = S_{n-1}^+ + \log [f(\bX_n, \bthetahat_{n-w,n})/f(\bX_n, \bm 0)]$ in \eqref{eq:WL_stat} with window size $w$. Since the random variable $\bthetahat_{n-w, n}$ is $\mathcal{F}_{n-1}$-measurable

\begin{equation}
    \Ex_\infty\left[ \frac{f(\bX_n, \bthetahat_{n-w,n})}{f(\bX_n, \bm 0)}\given[\Big]\calF_{n-1}\right] = \int_{\mathbb R^K}f(\bX_n, \bthetahat_{n-w,n})d\bX_n = 1,
\end{equation}
and therefore by iterated expectation and Jensen's inequality the log-likelihood ratio increment of $S_n$ has negative expectation under $\P_\infty$ for any estimator $\bthetahat$. After the change, the drift of $S_n$ is given by
\begin{align}
    \IHAT &\define \Ex_1^\btheta\left[\log\frac{f(\bX_n, \bthetahat_{n-w,n})}{f(\bX_n, \bm 0)}\right] \\
    &= \Ex_1^\btheta\left(\Ex_1^\btheta\left[\log\frac{f(\bX_n, \bthetahat_{n-w,n})}{f(\bX_n, \bm 0)}\given[\Big]\calF_{n-1}\right]\right) \\
    &= \Ex_1^\btheta\left((\bthetahat_{n-w,n})^\top\btheta - \frac{1}{2}\norm{\bthetahat_{n-w,n}}^2\right) \\
    &= \frac{1}{2} \Ex_1^\btheta\left(\thetanorm - \norm{\btheta- \bthetahat_{n-w,n}}^2\right) \\
    &= \KL - \frac{1}{2}\MSE{\bthetahat_{1,w+1}}\label{eq:IHAT_expression},
\end{align}
where the second to last equality is found by completing the square. Note that $\IHAT$ is a function of the window size $w$, although this is suppressed from the notation. A key observation from \eqref{eq:IHAT_expression} is that the post-change drift depends  on $\bthetahat$ exclusively through the mean square error and is maximized when the MSE is small. As discussed in the previous section, the James-Stein estimator dominates the usual ML-estimator in the considered setting in terms of MSE. 
Therefore, we propose to extend the WL-CuSum and SRRS tests by incorporating the James-Stein estimator in order to obtain significant improvements in the detection performance, as will be seen in the following sections.

\subsection{The JS-WL-CuSum test}\label{subsec:WL_CuSUm}
In this subsection we analyze the JS-WL-CuSum test, i.e. the WL-CuSum test of \eqref{eq:WL_stat}-\eqref{eq:WL_stop} where the James-Stein estimator is used as the estimator of $\btheta$. First, we establish a non-asymptotic upper bound on the detection delay of the WL-CuSum test when a generic estimator $\bthetahat$ is used as the parameter estimator. In the considered Gaussian case, the test is given by 
 \begin{align}
     S_n &:= S_{n-1}^+ + \bthetahat_{n-w, n}^\top \bX_n - \frac{1}{2}\norm{\bthetahat_{n-w, n}}^2 \label{eq:WL_Gaus_def}\\
     \TWL &= \inf\left\{n : S_n > b\right\}.
 \end{align}
It is seen that the upper bound depends on the chosen estimator only via its mean square error.
\begin{Theorem}\label{thm:WL_UB}
    Suppose that the window size $w$ is chosen to be large enough such that $\IHAT = \KL - \frac{1}{2}\MSE{\bthetahat_{1,w+1}} > 0$, and $b = \log \gamma$. Then, $\Ex_\infty(\TWL) \geq \gamma$ and 
    \begin{equation}\label{eq: WL_UB}
        \delay{\TWL} \leq \dfrac{b+(w+1)\KL +2}{\KL - \frac{1}{2}\MSE{\bthetahat_{1,w+1}}}.
    \end{equation}
\end{Theorem}
\begin{proof}
    The ARL lower bound follows directly from \cite[Lemma 3]{XIE_2023} where it is proven that $\Ex_\infty(\TWL) \geq e^b$ for any estimator $\bthetahat$ and threshold $b$. Therefore, choosing $b = \log \gamma$ is sufficient to guarantee $\Ex_\infty(\TWL)\geq b$. 
    
    For the upper bound \eqref{eq: WL_UB}, a sketch of the proof is as follows. First, we define an alternate statistic $U_n$ recursively by
    \begin{equation}
        U_n = U_{n-1} + \bthetahat_{n-w, n}^\top \bX_n - \frac{1}{2}\norm{\bthetahat_{n-w, n}}^2, \quad U_0,...,U_w =0,
    \end{equation}
    and the associated stopping time $T'_b = \inf\{n : U_n > b\}$, for which it can be shown that $\delay\TWL \leq \Extheta(T'_b)$. We then apply Wald's identity to the log-likelihood ratio process $\Lambda_{1,n}^\btheta$ defined in \eqref{eq: LLR_lambda_def} to obtain 
    \begin{align}
        \Extheta(T'_b)\KL &= \Extheta(\Lambda_{1,T'_b}) \\
        &=\Extheta(\Lambda_{1,T'_b}-U_{T_b'}) + \Extheta(U_{T'_b})\label{eq: sketch_wald}.
    \end{align}
    The result is obtained by finding appropriate upper bounds for the two terms in \eqref{eq: sketch_wald}. Complete details are given in Appendix~\ref{proof:WL_UB}.
\end{proof}
We note that an upper bound for the detection delay with general (not only Gaussian) pre- and post-change distributions has been derived previously in \cite{XIE_2023}. The bound derived above in Theorem~\ref{thm:WL_UB} is tighter than the general bound presented in \cite{XIE_2023}, as it is possible to control the overshoot $\Ex_1^\btheta(U_\TWL-b)$ by a small constant independent of $b$ in the Gaussian case. 

\begin{Remark}
    Due to the test statistic \eqref{eq:WL_Gaus_def} resetting at 0 and the dependent nature of the estimator sequence $\bthetahat_{1,w+1}, \bthetahat_{2,w+2},\ldots$, obtaining an exact non-asymptotic expression for the detection delay is difficult. In the proof of Theorem~\ref{thm:WL_UB}, we analyze a slightly modified stopping time $T'_b~=~\inf\{n : U_n > b\}$ where 
        $U_n = U_{n-1} + \bthetahat_{n-w, n}^\top \bX_n - \frac{1}{2}\norm{\bthetahat_{n-w, n}}^2, n > w,$
 for which it can be shown that $\delay\TWL\leq \Extheta(T'_b)$. The alternative statistic $U_n$ may be seen to correspond to the actual statistic $S_n$, with the distinction that $U_n$ does not restart at $0$. If the parameter estimates $\bthetahat_{1,w+1}, \bthetahat_{2,w+2},\ldots$ were independent, the increments of $U_n$ would constitute an i.i.d. sequence for $n > w$ with mean $\IHAT$. Then, assuming $\IHAT > 0$, Wald's equation \cite[Corollary 2.3.1]{TARTAKOVSKY_BOOK} could be directly applied to obtain
 \begin{equation}\label{eq:wald_approx}
     \Extheta(T'_b) = w+\frac{\Extheta(U_{T' _b})}{\IHAT} = w+\frac{b + \Extheta(U_{T'_b}-b)}{\IHAT},
 \end{equation}
 where the overshoot $\Extheta(U_{T'_b}-b)$ is generally small compared to $w$ or $b$. For instance, the proof of Theorem~\ref{thm:WL_UB} we show that $\Extheta(U_{T'_b}-b) \leq \KL + 2$. Although equation \eqref{eq:wald_approx} strictly applies only when $w=1$, (i.e. when the estimator sequence is in fact independent), we suggest employing it as a straightforward approximation for the detection delay of WL-CuSum across different values of $w$. Ignoring the overshoot, we obtain from \eqref{eq:wald_approx}
 \begin{equation}\label{eq:simple_approx}
     \delay\TWL \approx w + \frac{b}{\KL - \frac{1}{2}\MSE{\bthetahat_{1,w+1}}}.
 \end{equation}
 In Section~\ref{sec:sim}, Figure~\ref{fig:bound_eval} we demonstrate numerically that this approximation is accurate for practically important values of $w$.
\end{Remark}

The key takeaway from Theorem~\ref{thm:WL_UB} is that the estimator $\bthetahat$ affects the bound only through the mean square error. The same is true also for the simple approximation \eqref{eq:simple_approx}. As a result, due to its strong MSE properties, the James-Stein estimator is a natural choice as the estimator in the test. For instance, since the JS estimator dominates the maximum likelihood estimator in the MSE sense for all values of $\btheta$ and $w$ when $K\geq 3$, our results indicate that the JS estimator is a uniform improvement over the ML estimator in the WL-CuSum test in the considered settings. It should be mentioned that our results are based on upper bounds and approximations rather than exact expressions. Nonetheless, the results provide a robust indication of the JS estimator’s superior performance in this setting, which is confirmed in the simulations in Section~\ref{sec:sim}.

The upper bound in Theorem~\ref{thm:WL_UB} requires that $\IHAT > 0$, i.e. $w$ must be chosen such that $\MSE{\btheta_{1,w+1}} \leq 2\KL = \norm{\btheta}^2$. For example, for the JS estimator $\MSE{\thetaJS_{1,w+1}} < \frac{K}{w}$ \eqref{eq: JS_MSE}, and therefore choosing $w \geq K/\norm{\btheta}^2$ would ensure that $\IHAT > 0$. In particular, it is observed that the smaller the change, the larger the window must be. Of course, choosing $w \geq K\norm{\btheta}^2$ requires knowing at least some lower bound on $\norm{\btheta}^2$ which in general is unknown. Moreover, one can insert the MSE of the JS estimator (which depends on $w$ and $\btheta$) from \eqref{eq: JS_MSE} or \eqref{eq: JSV_mse} into the upper bound and solve for the value of $w$ which minimizes the upper bound. Again, the minimizer depends on $\btheta$, which is unknown. Therefore, in practice it is recommended to select a range of window values $w \in [1,W]$ for some large $W$, and run the WL-CuSum test simultaneously for each value, as proposed in \cite{XIE_2023}. For example, suppose there exists some value $\vartheta$ such that all changes for which $\norm{\btheta}^2 < \vartheta$ are not significant enough to warrant detecting. Then, choosing $W > K/\vartheta$ is sufficient for ensuring there exists a test with $\IHAT > 0$ for all $\btheta$ such that $\norm\btheta^2 \geq \vartheta$. 

As discussed in Section \ref{subsec:JS}, there exists an additional degree of freedom in the James-Stein estimator in choosing the shrinkage target. The target can be a single point $\bmu \in \bbR^K$ \eqref{eq: JS_definition} or more generally a subspace $\bbV \subset \bbR^K$ \eqref{eq: general_JS}. The closer the true $\bmu$ is to the target, the larger the reduction in MSE \eqref{eq: JSV_mse}. Naturally, if specific prior information regarding $\btheta$ exists, this information can be used in choosing the shrinkage target. An example is presented later in Section \ref{subsec:spat_example}. In the absence of further information regarding $\btheta$, two effective shrinkage targets are the pre-change mean $\bf 0$ or the "global-mean" \eqref{eq:global_mean JS}. When shrinking toward the pre-change mean, the JS estimator is most effective when $\norm\btheta^2 \propto \KL$ is small (see Fig. \ref{fig:mseplot} and \eqref{eq: JS_MSE}), i.e. when the change is small and the detection task is the most difficult. When $\norm{\btheta}^2$ grows, the gain in MSE diminishes, but also the change is easier to detect. Similar reasoning applies for the estimator $\thetaJSm$ \eqref{eq:global_mean JS} shrinking toward the global mean: if the components of $\btheta$ are not similar to each other, $\KL$ cannot be very small. Moreover, since $\thetaJSm$ is effective on a larger set of post-change parameter values than the zero-shrinking variant, we suggest using the global mean of the window as the default shrinkage target. Therefore, the proposed estimator to be inserted into \eqref{eq:WL_Gaus_def} is given by
\begin{equation}
    \bthetahat_{n-w, n} = \left(1 - \frac{(K-3)}{w\lVert {\barX_{n-w}^{n-1}}-
    \mathbf{m}\rVert^2}\right)^+({\barX_{n-w}^{n-1}}-
    \mathbf{m}) + \mathbf{m},
\end{equation}
where $\barX_{n-w}^{n-1}$ is the component-wise mean vector of $\bX_{n-w},\ldots,\bX_{n-1}$ and $\mathbf{m} = (1/K)\mathbf{1}^\top\barX_{n-w}^{n-1}$ is its empirical mean.

\color{black}
\subsection{The JS-SRRS test}
In the previous section, we showed that the James-Stein estimator can be employed to obtain a uniform non-asymptotic improvement on the MLE when used with the finite-window WL-CuSum test. Next, we will show that when used in conjunction with the infinite-window SRRS test, a James-Stein-based test is asymptotically superior to the GLR-CuSum test \eqref{eq:GLR_stat} in a well-defined sense. 

The SRRS test $\TSRRS$ with a general estimator $\bthetahat$ is given by \cite{LORDEN_2005}
\begin{align}
    \widehat\Lambda_{t,n} &\define   \sum_{m = t}^n\left( {(\bthetahat_{t,m})}^\top \bX_m - \frac{1}{2} \norm{\bthetahat_{t,m}} ^2\right) \label{eq: SRRS_estimator}\\
    \TSRRS &\define \inf\left\{n : \sum_{t = 1}^n e^{\widehat\Lambda_{t, n}} > e^b\right\}.
\end{align}
The theorem below relates the MSE of the used estimator to the asymptotic detection delay of the SRRS test.
\begin{Theorem}\label{theorem:SRRS}
    Let the mean square error of the estimator $\bthetahat$ satisfy
    \begin{equation}\label{eq:MSE_condition}
        \MSE{\bthetahat_{1, n+1}} \leq \frac{\kappa(\btheta)}{n}, \quad n \geq 1.
    \end{equation}
    for some value $\kappa(\btheta)$ which may depend on $\btheta$.
    Then, for $b =  \log \gamma$ the SRRS test satisfies $\Ex_\infty(\TSRRS) \geq \gamma$ and
    \begin{equation}
        \delay\TSRRS \leq \frac{\log \gamma}{\KL} + \frac{\kappa(\btheta)\log\log\gamma}{2\KL} + O(1), \quad \gamma \to \infty.
    \end{equation}
\end{Theorem}
\begin{proof}
The proof is given in Appendix \ref{proof:SRRS_theorem}.
\end{proof}

In our proposed JS-SRRS test the estimator $\bthetahat_{t,n}$ in \eqref{eq: SRRS_estimator} has the general form
\begin{equation}\label{eq: SRRS_JS_estimator}
    \bthetahat_{t,n} = \left(1 - \frac{(n-t-1)^{-1}(K-d-2)}{\lVert {\overbar\bX_t^{n-1}}-
    \projV(\overbar\bX_t^{n-1})\rVert^2}\right)^+({\overbar\bX_t^{n-1}}-
    \projV(\overbar\bX_t^{n-1})) + \projV(\overbar\bX_t^{n-1}),
\end{equation}
where $\bbV$ is a target subspace of $\mathbb R^K$ with dimension less than $K-2$ and $\projV(\bX)$ denotes the projection of $\bX$ onto $\bbV$, see Remark \ref{remark:global_mean}.
\begin{Corollary}\label{cor: SRRS_general}
 Let $\bbV$ be a subspace of $\mathbb{R}^K$ with $d\define \dim{\bbV} < K-2$. If $\bthetahat$ is chosen as in \eqref{eq: SRRS_JS_estimator}, then as $\gamma \to \infty$   
  \begin{equation}\label{eq: JS_SRRS_delay}
        \delay\TSRRS=\begin{cases}
            \KL^{-1}\left(\log \gamma + (d+2)/2\log \log \gamma\right) + O(1) \quad &\emph{if } \btheta \in \bbV\\
             \KL^{-1}\left(\log \gamma + (K/2)\log \log \gamma \right) + O(1) &\emph{otherwise.} \\
        \end{cases} 
    \end{equation}
    Therefore, 
    \begin{equation}
       \sup_{\btheta}\KL \delay\TSRRS - \inf_{T \in \calC_\gamma}\sup_{\btheta}\KL \delay T = O(1)
    \end{equation}
    meaning that the JS-SRRS test is second-order asymptotically minimax for any shrinkage target $\bbV$ as $\gamma \to \infty$.
\end{Corollary}
\begin{proof}
    By \eqref{eq: JSV_mse} and \eqref{eq: JSV_mse_small}, the estimator in \eqref{eq: SRRS_JS_estimator} achieves \eqref{eq:MSE_condition} with $\kappa(\btheta) = d$ if $\btheta \in \bbV$ and $\kappa(\btheta) = K$ otherwise. The expression for the detection delay then follows from Theorem~\ref{theorem:SRRS}. Second-order asymptotic minimaxity stems from observing that \eqref{eq: JS_SRRS_delay} matches the lower bound in Lemma~\ref{lemma: inf_LB} up to $O(1)$.
\end{proof}

Comparing the result in Corollary~\ref{cor: SRRS_general} with the asymptotic detection delay of the GLR test \eqref{eq:GLR_delay} shows that, up to constant terms, the JS-SRRS test $\TSRRS$ is never worse than the GLR test $\TGLR$, and it is strictly better for some values of $\btheta$. In particular, if $\btheta$ lies in the target subspace $\bbV$, then $\delay\TGLR - \delay\TSRRS = O((K-d)/2\log \log \gamma) \to \infty$ as $\gamma \to \infty$, whereas for other $\btheta$, the first two asymptotic terms match. Especially for large $K$, the improvement in the second-order term can be significant for practical finite values of $\gamma$ as will be demonstrated in Section~\ref{sec:sim}. Corollary \ref{cor: SRRS_general} also shows that the James-Stein estimator presents an improvement over the linear shrinkage estimators considered in \cite{WANG_2015}, as it can improve the second-order term without sacrificing first-order optimality for any $\btheta$. 
\begin{Remark}
    It is interesting to note that the asymptotic properties of the JS-SRRS test stem from the fact that the James-Stein estimator is a \emph{superefficient} estimator \cite{VANDERVAART}. Under squared error loss, an estimator $\bthetahat$ is said to be superefficient if for all $\btheta$ \cite[pp. 440]{LEHMANN_BOOK}
    \begin{equation}
        \lim_{n\to\infty} n\MSE{\bthetahat_{1,n+1}} \leq \mathsf{Tr}\left(\mathsf{FI}^{-1}(\btheta)\right) = K,
    \end{equation}
    where $\mathsf{FI}(\btheta)$ denotes the Fisher information matrix of a single observation, and the inequality is strict for at least one $\btheta$. Points where the inequality is strict are called points of superefficiency, which for the ordinary JS estimator occur at $\btheta = \bmu$, and for the subspace-shrinking JS at $\btheta \in \bbV$. Observe, that the definition of superefficiency is equivalent to an estimator satisfying the MSE condition of Theorem \ref{theorem:SRRS} \eqref{eq:MSE_condition} with $\kappa(\btheta) \leq K$. As the set of points of superefficiency cannot have Lebesgue measure greater than zero \cite{VANDERVAART,LECAM}, Theorem \ref{theorem:SRRS} implies that the SRRS test with any estimator cannot improve on the second-order detection delay term except on a set of zero measure. The JS-SRRS test achieves the improvement in the $d$-dimensional subspace $\bbV$. 
\end{Remark}

\section{Simulations} \label{sec:sim}
In this section, we evaluate the proposed tests and validate the analytical results in simulations. In all experiments, the global mean positive-part James-Stein estimator 
\begin{equation}\label{eq:sim_global_JS}
    \thetaJSmplus(\bX) \define \left(1 - \frac{(K-3)}{w\lVert {\bX}-
    \mathbf{m}\rVert^2}\right)^+({\bX}-
    \mathbf{m}) + \mathbf{m}
\end{equation}
where $\mathbf{m} = m\mathbf{1}$ and $m$ is the empirical mean of the components of $\bX$, is employed as the specific variant of the JS estimator incorporated into the proposed tests. As mentioned in Remark~\ref{remark:global_mean}, this estimator is a special case of the general subspace-shrinking James-Stein estimator with $\bbV = \text{span}(\bm 1)$. As recommended in \cite{XIE_2023}, we run the parallel variant of the WL-CuSum test in place of the fixed window WL-CuSum test. The parallel WL-CuSum test is defined by simultaneously running $W$ WL-CuSum tests with window sizes ranging from 1 to $W$ and stopping when the first individual test detects a change. That is, if $\TWL(w)$ denotes the WL-CuSum test of \eqref{eq:WL_stop} with window size $w$, the parallel WL-CuSum stopping time is given by $\min_{1\leq w \leq W}\TWL(w)$. The maximal window size is chosen as $W = 200$ for all experiments. Supported by the results of Section~\ref{subsec:WL_CuSUm} JS estimator reduces the detection delay of $\TWL(w)$ for all $w$, and therefore we expect it to reduce the detection delay of the parallel test as well. The proposed JS-WL-CuSum and JS-SRRS tests are compared against their maximum-likelihood counterparts (ML-WL-CuSum and ML-SRRS), as well as the GLR test \eqref{eq:GLR_stat}. The GLR test is computationally expensive, and therefore we apply a window-limited version \cite{LAI_1998} with a window size of 200. In our experiments, the system parameters are chosen such that expected detection delay $<< 200$ to ensure that the windowing does not meaningfully affect detection performance. 

{\color{black}
\subsection{Evaluation of the upper bound}
First, the accuracies of the upper bound derived in Theorem~\ref{thm:WL_UB} and the approximation proposed in \eqref{eq:simple_approx} for the WL-CuSum test are evaluated. Since the analytical results are for the ordinary WL-CuSum with a single window, not the parallel variant, in Figure~\ref{fig:bound_eval} the tests are also run with a single window. The true $\btheta$ is chosen as $[K^{-1/2},\ldots,K^{-1/2}]$ for $K = 10$ (left) and $K =50$ (right). For both tests, the window sizes that minimize the approximation in \eqref{eq:simple_approx} at $b=10$ are chosen. The observed detection delays for a range of values of $b$ (solid lines) are compared with the derived upper bounds (dotted lines) and the proposed approximations (dot-dash lines). The bounds are observed to be reasonably tight, especially with the JS estimator. Moreover, the approximations accurately match the true observed detection delays.

\begin{figure}
\centering
\includegraphics{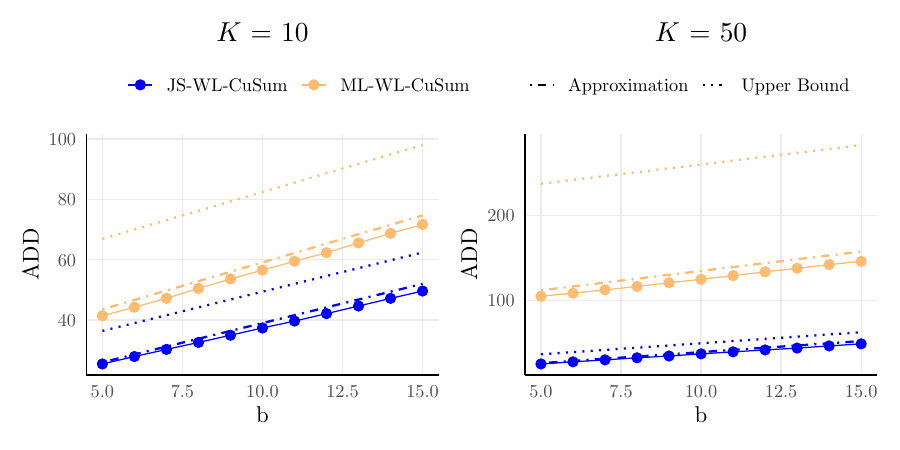}
\caption{The upper bounds of Theorem~\ref{thm:WL_UB} (dotted lines) and approximation in eq. \eqref{eq:simple_approx} compared against the true detection delay for a range of values of $b$. Both the upper bound and the approximation are smaller when using the JS estimator (blue) compared to the ML estimator (yellow). It is also seen that in particular when using the JS estimator, the analytical results accurately reflect the true detection delay (solid lines).}
\label{fig:bound_eval}
\end{figure}
}

\subsection{Change affects all streams}
\begin{figure}
    \centering
    \includegraphics{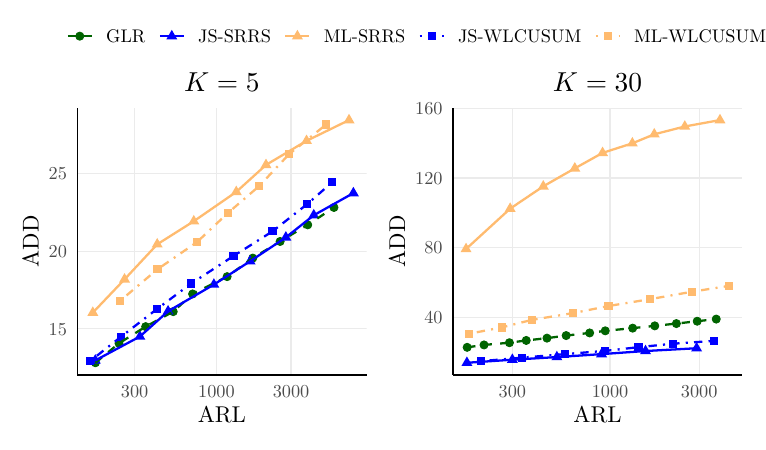}
    \caption{Tradeoff between the average run length to false alarm (ARL) and detection delay (ADD) for the studied tests with $K=5$ (left) and $K=10$ (right). James-Stein variants (blue) uniformly improve on the ML versions (gold). For $K=30$, the proposed tests are also significantly better than the GLR test \cite{LAI_1998} in terms of reduced detection delay.}
    \label{fig:arl_add}
\end{figure}
We consider a setting where the post-change mean vector $\btheta$ is given by $\Tilde{\btheta}/\norm{\Tilde{\btheta}}$ where $\Tilde{\btheta} = [1,2,...,K]^\top$. This depicts a scenario where the change affects all $K$ sensors in the network simultaneously with different intensities, with the total magnitude $\norm \btheta$ always equal to 1. In Figure \ref{fig:arl_add}, we plot the tradeoff between ARL and average detection delay (ADD) for the compared tests for $K = 5$ and $K = 30$. It is immediately observed that the James-Stein variants of the WL-CuSum and SRRS tests (in blue) perform noticeably better than the versions using maximum likelihood estimators (in gold). When $K = 5$, the JS-based tests perform comparably to the GLR test, but for $K = 30$ there exists a noticeable gap in favor of the proposed tests. The effect of the dimension $K$ is further illustrated in Figure \ref{fig:K_ADD}, where we fix $\gamma = 2000$, and vary $K$ from 5 to 50. Even though the detection problem should become more difficult as $K$ increases (as can be observed from the ML-based tests), the tests utilizing the James-Stein estimator suffer essentially no degradation in detection delay. This stems from the fact that shrinkage estimation is in general more powerful the higher the dimension. Therefore, utilizing the JS estimator is especially effective in large-scale problems. The JS-SRRS test performs slightly better than the JS-WL-CuSum test, which is perhaps not surprising since the SRRS test is also computationally more intensive than WL-CuSum. 

ML-SRRS \cite{LORDEN_2005} is seen to suffer drastically as $K$ increases. The ML-estimator requires a window size of at least $K/\norm{\btheta}^2 (=K$ here as $\norm\btheta^2=1$) to ensure a positive expectation for the summands in (13). In the post-change regime, for $n < t+K$, the statistics $\widehat\Lambda_{t,n}$ have a negative drift and unlike in WL-CuSum, the statistics do not restart at 0. As a result, the values can be well below zero before starting to grow.

\begin{figure}
    \centering
    \includegraphics{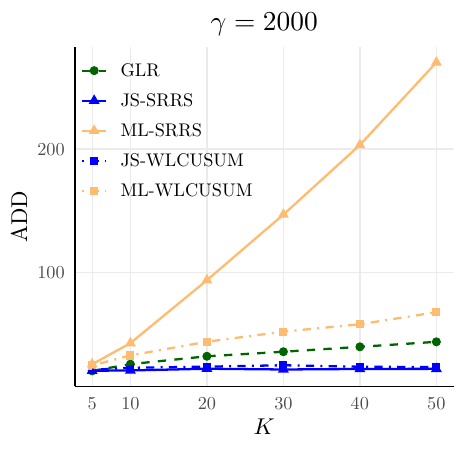}
    \caption{Detection delay against the dimension $K$, when $\norm\btheta = 1$ and $\gamma = 2000$. It is observed, that the higher the dimension, the larger the gap in performance in favor of the James-Stein tests. Hence, JS-based test scales well for larger number of streams and sensors.}
    \label{fig:K_ADD}
\end{figure}
\subsection{Sparse change}
Next, we study the case where the change affects only a subset of the data streams. We choose $\gamma = 2000, K = 20$ and $\btheta$ to be of the form $[\underset{k}{\underbrace{k^{-1/2},\ldots, k^{-1/2}}},0,\ldots,0]$ for $k \in [1,20]$ so that again $\norm\btheta = 1$ no matter how many streams $k$ are affected. The results are shown in Figure \ref{fig:varying_k}. For very sparse changes, the GLR test is superior to the proposed tests, but the roles are reversed when approximately 40\% or more of the streams are affected. For any number of streams affected, the WL-CuSum and SRRS tests are drastically improved by James-Stein estimation, corroborating the analytical results of previous sections. We note that only in the case where all $K=20$ data streams are affected does $\btheta$ lie in the prespecified target subspace $\bbV$ of the used James-Stein estimator. Hence, substantial gains can be achieved even when $\btheta \not\in \bbV$.
\begin{figure}
    \centering
    \includegraphics{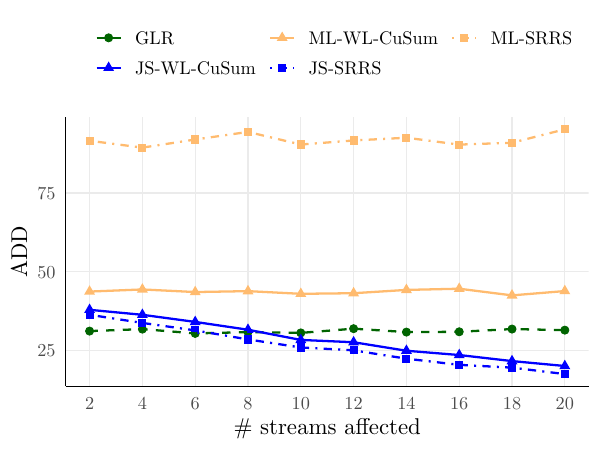}
    \caption{Average detection delay when the change affects a varying number of sensors out of 20, with total signal magnitude $\norm\btheta$ fixed at 1. JS-based tests are uniformly better than ML-alternatives and better than the GLR test when more than 40\% of the streams are affected.}
    \label{fig:varying_k}
\end{figure}

\color{black}
\subsection{Change detection with spatial structure}\label{subsec:spat_example}
As discussed earlier, the shrinkage target $\bbV$ used in the James-Stein estimator can theoretically be any linear subspace of $\mathbb{R}^K$ so long as $\dim(\bbV) < K -2$. In this example, we highlight how information regarding the structure of $\btheta$ can be used in choosing $\bbV$. Consider a system with $K$ sensors at locations $s_1,\ldots, s_K$ and $M < K-2$ sources located at $z_1, \ldots, z_M$. Suppose that the post-change signal magnitude $\btheta^{(k)}$ at the $k$th sensor is given by
\begin{equation}
    \btheta^{(k)} = \sum_{m=1}^M\beta_m h(s_k, z_m),
\end{equation}
where $\beta_m$ is an unknown scalar parameter representing the signal strength of the $m$th source, and $h(s,z)$ is a function describing the attenuation of the signal transmitted from $z$ at $s$. Generally, $h$ is some decreasing function of $\norm{s-z}$, for example $h(s,z) = C\norm{s-z}^{-2}$ in the case of electromagnetic waves propagating in free space, for some constant $C$. Under this model, the post-change mean can be written as $\btheta = \bZ\bbeta$ where $\mathbf{Z} \in \mathbb{R}^{K \times M}$ such that $\bZ_{km} = h(s_k, z_m)$ and $\bbeta = [\beta_1, \ldots, \beta_M]^\top \in \bbR^M$. The maximum likelihood estimate of $\btheta$ conditional on observing $\bX\sim \calN(\btheta, \bI)$ given the model is equal to $\thetaML_\bZ(\bX)~=~\bZ (\bZ^\top\bZ)^{-1}\bZ\bX$. This estimator could in principle be used as the estimator in the considered WL-CuSum or SRRS tests. However, such an estimator may not be robust to deviations from model assumptions, as can happen if, for example, the true source locations differ from the assumed \cite{ROBUST_BOOK}.

As an alternative, one can make use of the prior information by using the column space of $\bZ$ as the shrinkage target subpace $\bbV$ in the James-Stein estimator. Observing that with this choice of $\bbV$, the projection of $\bX$ onto $\bbV$ is given by $\projV(X) = \bZ (\bZ^\top\bZ)^{-1}\bZ\bX = \thetaML_\bZ(\bX)$, the JS estimator is given by
\begin{equation}\label{eq: sim_JSV}
    \thetaJSV(\bX) \define \left(\frac{w^{-1}(K-M-2)}{\norm{\bX - \thetaML_\bZ(\bX)}^2}\right)^+(\bX - \thetaML_\bZ(\bX)) + \thetaML_\bZ(\bX).
\end{equation}
We expect this estimator to estimate $\btheta$ effectively when $\btheta \approx \bZ\bbeta$, but still provide performance superior to the unconstrained ML-estimator $\thetaML(\bX) = \bX$ even if the assumed model is incorrect.

We consider an example with $K = 20$ uniformly spaced sensors on $[0, 100]$, that is, $s_k = 100(k-1)/(K-1), k \in \{1,\ldots,K\}$ \cite{HALME_TSP2022}. There are $M=2$ sources that are assumed to be located at $z_1 = 20$ and $z_2 = 80$. The signal is considered to attenuate in proportion to the inverse of the squared distance, i.e. $h(s,z) = \min(1,\norm{s-z}^{-2}).$ The source signal is set as $\bbeta = [\beta, \beta]^\top$, with $\beta$ is chosen such that $\norm\btheta = \bZ\bbeta = 1$. The WL-CuSum test is considered with four different estimators: 1) the global-mean shrinking JS estimator \eqref{eq:sim_global_JS} used in previous experiments, 2) subspace shrinking estimator $\thetaJSV$ defined in \eqref{eq: sim_JSV} which uses the spatial information in choosing the shrinkage target, 3) the ML-estimator under the assumed model $\thetaML_\bZ(\bX)$, which is also the least squares (LS) estimator, and 4) the ordinary maximum likelihood estimator that ignores the spatial model $\thetaML(\bX) = \bX$. Additionally, the GLR test is included as a comparison.

\begin{figure}

    \centering
    \includegraphics[width=0.7\linewidth]{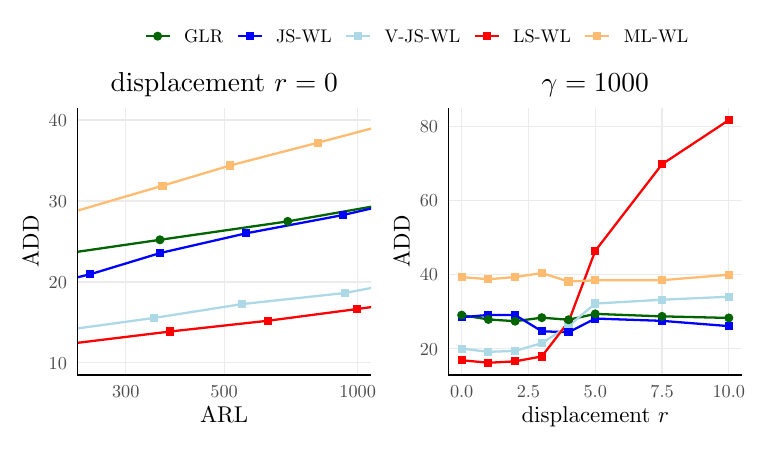}
    \caption{Left: Performance of the tests for various ARL levels under correct model assumptions. The LS-estimator (red) and $\thetaJSV$ (light blue) show the best detection performance. Right: Detection delay as a function of source displacement parameter $r$ for a fixed ARL level ($\gamma = 1000$). The LS-estimator performs best for small $r$, but its performance degrades with larger displacements. The detection delay of $\thetaJSV$ increases modestly with $r$, remaining bounded by the ordinary ML-based test (yellow).}
    \label{fig:spatial_sim}
\end{figure}

In the left plot of Figure \ref{fig:spatial_sim}, the performance of the tests is evaluated for a range of ARL levels when the above mentioned model assumptions hold. As may be expected, the two methods that use the specific data model, i.e. the LS-estimator (red) and $\thetaJSV$ (light blue) provide the best detection performance. In the right plot of Fig. \ref{fig:spatial_sim}, the performance of the methods is evaluated when the model assumptions are slightly violated. In particular, instead placing the sources at $s_1 = 20$ and $s_2 =80$, the sources are placed at $20+r$ and $80-r$, respectively, for various values of the displacement parameter $r$. The detection delay of the methods is displayed as a function of $r$ for a fixed ARL level of $\gamma = 1000$. When $r$ is small, meaning that the assumed model is accurate, the LS-estimator performs best, followed closely by the JS estimator that utilizes the signal model. However, once the source displacement is large enough, the LS-estimator-based test's performance rapidly degrades. The detection delay of $\thetaJSV$ also increases as the displacement increases, but crucially, the decrease is modest. Indeed, a large violation of the model assumptions only means that the chosen shrinkage target is unhelpful. The detection delay is still upper bounded by the delay of the ordinary ML-based test (yellow), as the analytical results of Section \ref{sec: main} indicate. The results demonstrate that the JS estimator with a shrinkage target chosen using system information offers a good trade-off between performance and robustness. 

\color{black}
\section{Conclusion}\label{sec: conclusion}
In this paper, we proposed and analyzed procedures for quickest detection of a mean-shift in parallel Gaussian data streams. A non-asymptotic upper bound on the detection delay of the WL-CuSum test was derived, and it was observed that the mean square error of the chosen estimator critically affects the ensuing detection delay. The JS-SRRS test was shown to be second-order asymptotically minimax, with the second-order term further improved if the true parameter value lies in the shrinkage target subspace. In simulations it was demonstrated that the proposed James-Stein estimator-based procedures achieve significant reductions in detection delay compared to maximum likelihood alternatives, especially if the problem is high-dimensional and the true parameter value is close to the chosen shrinkage target subspace. The suggested tests perform favorably compared with the GLR-CuSum test as well. Our analytical and experimental results indicate that replacing the conventional maximum likelihood estimator with the JS estimator in the suggested tests essentially enables the incorporation of prior information about the true parameter value through a shrinkage target. 
Even if the target is entirely irrelevant, the detection delay does not exceed what can be achieved using the maximum likelihood estimator.

{\color{black}
In this paper, only the Gaussian mean-shift setting with independent observations and diagonal covariance was explicitly considered. While seemingly simplistic, we stress that many more complicated Gaussian data models can be converted to this setting through appropriate transformations. For instance, cases with a known covariance can be handled by whitening the data, and time-dependent observations can similarly be transformed to an equivalent mean-shift setting by decorrelating the time series. These pre-processing steps allow the methodology to be applied in a much broader range of scenarios without fundamental modifications.

In future work, it would be of interest to extend the results of this paper to detecting changes in parameters other than the Gaussian mean, for example the covariance matrix \cite{XIE_SUBSPACE, CHAUDHURI_2024}. Moreover, extension to   probability models other than Gaussian would be of interest. Shrinkage estimators with domination properties similar to the JS estimator have been developed for more general classes of probability models, such as elliptically symmetric distributions \cite{FOURDRINIER_BOOK, BRANDWEIN_1993}. Extending the results of this paper to such models is an important topic of future research.
}

\section*{Acknowledgments}
The authors would like to acknowledge discussions with George Moustakides regarding the window-limited CuSum test.

\appendices

{\color{black} \section{Proof of Theorem \ref{thm:WL_UB}}\label{proof:WL_UB}
    Define the statistic $U_n$ recursively by
    \begin{equation}
        U_n = U_{n-1} + \bthetahat_{n-w, n}^\top \bX_n - \frac{1}{2}\norm{\bthetahat_{n-w, n}}^2, \quad U_0,...,U_w =0,
    \end{equation}
    and the associated stopping time $T'_b = \inf\{n : U_n > b\}$. The definition of $U_n$ is equal to the WL-CuSum test statistic \eqref{eq:WL_Gaus_def} but without the positive-part operation. As such, it is easy to see that $S_n \geq U_n$ and therefore $\TWL \leq T'_b$ for all $b$. Moreover, by \cite[Lemma 4]{XIE_2023} $\delay{\TWL} \leq \Extheta(\TWL)$, i.e. the worst-case delay for the WL-CuSum test corresponds to the case when the change is present from the beginning. Therefore, to obtain an upper bound on $\delay\TWL$ it suffices to bound $\Extheta(T'_b).$ \\
    Recall, that 
    \begin{equation}
    \Lambda^{\btheta}_{1,n} \define \sum_{m = 1}^n \log \frac{f(\bX_m, \btheta)}{f(\bX_m, \bm 0)} = \sum_{m = 1}^n\left( {\btheta}^\top \bX_m - \frac{1}{2} \norm{\btheta} ^2\right)
    \end{equation}
    denotes the log-likelihood ratio between $\P_1^\btheta$ and $\P_\infty$ at time $n$.
    In what follows we write $\Lambda^{\btheta}_{n} =  \Lambda^{\btheta}_{1,n}$ and $\bthetahat_n= \bthetahat_{n-w,n}$ for brevity.
    Since $\Lambda^{\btheta}_{n}$ is a sum of $n$ i.i.d. random variables with mean $\KL$ under $\P_1^\btheta$ \eqref{eq: KL_def}, by Wald's equation 
    \begin{align}
        \Extheta(T')\KL &= \Extheta(\Lambda^{\btheta}_{T'}) \\
        &= \Extheta(U_{T'}) + \Extheta(\Lambda^{\btheta}_{T'} - U_{T'}).\label{eq: wald_decomposition}
    \end{align}
    The first term can be bounded by bounding the overshoot of the statistic $U_{T'}$ over the stopping threshold $b$. We make use of the fact that for a random variable $Z \sim \calN(\mu, \tau^2)$,
    \begin{equation}\label{eq: Gaus_resticted_mean}
        h(t) \define \Ex(Z + t \given Z  > -t) = \mu +t + \tau\frac{\varphi((t-\mu)/\tau)}{1-\Phi((t-\mu)/\tau)},
    \end{equation}
    where $\varphi$ and $\Phi$ denote the pdf and cdf of a standard Gaussian distribution, respectively. The function $\frac{\varphi(x)}{1-\Phi(x)}$ is the reciprocal of the Mills ratio \cite{GASULL_MILLS}, for which we need the upper bound \cite{SAMPFORD}
    \begin{equation}\label{eq: mills_bound}
        \frac{\varphi(x)}{1-\Phi(x)} \leq \frac{2}{\sqrt{x^2+4}-x} = \frac{\sqrt{x^2+4}+x}{2}.
    \end{equation}Writing $r \define U_{T'-1} -b$ and $\ell_n \define \bthetahat_{n}^\top \bX_{n} - \frac{1}{2}\norm{\bthetahat_{n}}^2$, we have
    \begin{align}
        \Extheta(U_{T'} -b) &= \Extheta(r + \ell_{T'}) =\Extheta(\Extheta(r + \ell_{T'} \given \calF_{T'-1}, T)) \\
         &\leq \Extheta\left(\sup_{r\leq0}\left( r + \Extheta(\ell_{T'} \given \bthetahat_{T'-1}, \ell_{T'} > -r)\right)\right) \\
         &= \Extheta\left(\Extheta(\ell_{T'} \given \bthetahat_{T'}, \ell_{T'} > 0)\right),
    \end{align}
    where the inequality holds since $r$ is $\calF_{T'-1}$-measurable and $\ell_{T'}$ depends on $\calF_{T'-1}$ only via $\bthetahat_{T'}.$ The final equality follows from the fact that $\ell_{T'} \given \bthetahat_{T'}$ is Gaussian and $h(t)$ \eqref{eq: Gaus_resticted_mean} is increasing in $t$, which may be verified by observing that $h'(t) > 0$.

    Since $\ell_{T'} \given \bthetahat_{T'} \sim \calN(\bthetahat_{T'}^\top\btheta - \frac{1}{2}\norm{\bthetahat_{T'}}^2, \norm{\bthetahat_{T'}}^2)$, we have using \eqref{eq: Gaus_resticted_mean} and \eqref{eq: mills_bound}
    \begin{align}
        \Extheta(\ell_{T'} \given \bthetahat_{T'}, \ell_{T'} > 0) &\leq \norm{\bthetahat_{T'}}g(\bthetahat) + \frac{\norm{\bthetahat_{T'}}}{2}\left(\sqrt{g(\bthetahat)^2 + 4} - g(\bthetahat)\right) \\
        &= \frac{\norm{\bthetahat_{T'}}}{2}\left(g(\bthetahat) + \sqrt{g(\bthetahat)^2 + 4}\right)\label{eq: Ex_using_g}
    \end{align}
    where
    \begin{align}
        g(\bthetahat) = \frac{\bthetahat_{T'}^\top\btheta - \frac{1}{2}\norm{\bthetahat_{T'}}^2}{\norm{\bthetahat_{T'}}}.
    \end{align}
    It is seen that \eqref{eq: Ex_using_g} is increasing in $g$, and $g(\bthetahat) \leq \norm{\btheta} - \frac{1}{2}\norm{\bthetahat_{T'}}$ by the Cauchy-Schwarz inequality. Inserting the upper bound for $g(\bthetahat)$ into \eqref{eq: Ex_using_g} we obtain
    \begin{equation}\label{eq: E_l_scalar}
        \Extheta(\ell_{T'} \given \bthetahat_{T'}, \ell_{T'} > 0) \leq \frac{\norm{\bthetahat_{T'}}}{2}\left(\norm{\btheta} - \frac{1}{2}\norm{\bthetahat_{T'}} + \sqrt{\left(\norm{\btheta} - \frac{1}{2}\norm{\bthetahat_{T'}}\right)^2+4}\right),  \end{equation}
    which depends on $\bthetahat_{T'}$ only via the scalar $\norm{\bthetahat_{T'}}$. Somewhat tedious but elementary calculus shows that the value of $\norm{\bthetahat_{T'}}$ that maximizes \eqref{eq: E_l_scalar} is $\norm{\bthetahat_{T'}} = \norm\btheta + 4/\norm\btheta$. Inserting this back into \eqref{eq: E_l_scalar} and rearranging we obtain
    \begin{align}
        \Extheta(\ell_{T'} \given \bthetahat_{T'}, \ell_{T'} > 0) &\leq \frac{1}{2}\left(\norm\btheta^2 + 4\right) \\
        &= \KL + 2
    \end{align}
    and therefore
    \begin{align}
        \Extheta(U_{T'} -b) &\leq \Extheta\left(\Extheta(\ell_{T'} \given \bthetahat_{T'}, \ell_{T'} > 0)\right) \\
        &\leq \Extheta\left(\KL + 2\right) = \KL + 2.\label{eq: overshoot_final}
    \end{align}
         
    For the second term in \eqref{eq: wald_decomposition}, since $U_1,...,U_w = 0$
    \begin{align}
        \Extheta(\Lambda^{\btheta}_{T'} - U_{T'}) &= \Extheta(\Lambda^\btheta_w -U_{w}) + \Extheta(\Lambda^\btheta_{T'}-\Lambda^\btheta_w -U_{T'} + U_w)\\
        &= w\KL + \Extheta\left(\sum_{n=w+1}^{T'} {\btheta}^\top \bX_n - \frac{1}{2} \norm{\btheta}^2-{\bthetahat_n}^\top \bX_n - \frac{1}{2} \norm{\bthetahat_n}^2\right) \\
        &= w\KL + \frac{1}{2}\Extheta\left(\sum_{n=w+1}^{T'}\norm{\bthetahat_n - \btheta}^2\right).\label{eq:ex_diff_intermediate}
    \end{align}
    Since the estimator sequence $\bthetahat_{w+1}, \bthetahat_{w+2}, \ldots$ is dependent, and $T'$ depends on this process, the sum on the last line is difficult to analyze precisely. An upper bound is found by utilizing the fact that the terms in the process $w$-dependent, i.e. $\bthetahat_{n+w}$ is independent of $\bthetahat_{n}$. Then, using a technique as in \cite{XIE_2023}
    \begin{align}
        \Extheta\left(\sum_{n=w+1}^{T'}\norm{\bthetahat_n - \btheta}^2\right) &\leq \Extheta\left(\sum_{n=w+1}^{T'+w}\norm{\bthetahat_n - \btheta}^2\right) \\
        &= \Ex\left(\sum_{n=w+1}^\infty \norm{\bthetahat_n - \btheta}^2\mathbf{1}_{\{T' \geq n-w\}}\right) \\
        &= \Ex\left(\sum_{n=w+1}^\infty \Extheta(\norm{\bthetahat_n - \btheta}^2|\calF_{n-w-1})\mathbf{1}_{\{T' \geq n-w\}}\right) \\
        &= \Ex\left(\sum_{n=w+1}^\infty \Extheta\norm{\bthetahat_n - \btheta}^2\mathbf{1}_{\{T' \geq n-w\}}\right)\\
        &= \MSE{\bthetahat_{w+1}}\Extheta(T'). \label{eq:mse_sum_final}
    \end{align}
    and hence
    \begin{equation}
        \Extheta(\Lambda^{\btheta}_{T'} - U_{T'}) = w\KL+\MSE{\bthetahat_{w+1}}\Extheta(T').\label{eq: ex_diff_complete}
    \end{equation}
    Inserting \eqref{eq: overshoot_final} and \eqref{eq: ex_diff_complete} into \eqref{eq: wald_decomposition} gives
    \begin{align}
        \Extheta(T')\KL \leq w\KL+\MSE{\bthetahat_{w+1}}\Extheta(T') +  \KL + 2.
    \end{align}
    Rearranging and dividing by $\KL - \MSE{\bthetahat_{w+1}}$, which is positive by assumption, yields
    \begin{equation}
        \delay{\TWL} \leq \Extheta(T') \leq \dfrac{b+(w+1)\KL +2}{\KL - \MSE{\bthetahat_{w+1}}}.
    \end{equation}
}

\section{Proof of Theorem \ref{theorem:SRRS}}\label{proof:SRRS_theorem}
First, we prove that $\Ex_\infty(\TSRRS) \geq e^b$. Denoting the test statistic sequence by $\{R_n\}$, i.e. $R_n = \sum_{t = 1}^n e^{\widehat\Lambda_{t, n}}$, it is observed that 
\begin{align}
    \Ex_\infty\left(R_n | \mathcal{F}_{n-1}\right) &= \Ex_\infty\left(e^{\widehat\Lambda_{n,n}} + \sum_{t=1}^{n-1}e^{\widehat\Lambda_{t, {n-1}}}\cdot\frac{f(\bX_n,\bthetahat_{t,n})}{f(\bX_n, \bm 0)} \given[\Big] \mathcal{F}_{n-1})\right) \\
    &= 1 + \sum_{t=1}^{n-1}e^{\widehat\Lambda_{t, {n-1}}}\underset{1}{\underbrace{\Ex_\infty\left[\frac{f(\bX_n,\bthetahat_{t,n})}{f(\bX_n, \bm 0)}\given[\Big] \mathcal{F}_{n-1}\right]}} \\
    &= 1 + R_{n-1}.
\end{align}
Therefore $\{R_n - n\}$ is a $\P_\infty$-martingale, and by the optional sampling theorem \cite[Th. 2.3.2]{TARTAKOVSKY_BOOK} $\Ex_\infty(\TSRRS) = \Ex_\infty(R_\TSRRS) \geq e^b$. A similar argument for bounding the ARL of a change detection procedure is found in many references and textbooks, see e.g. \cite[8.2.1]{TARTAKOVSKY_BOOK}.

For the detection delay, observe that the worst case detection delay for the SRRS test occurs when the change happens at $\nu = 1$, i.e. $\delay\TSRRS = \Ex_1^\btheta\left(\TSRRS\right)$. In what follows we write $\TSRRS = T$, $\Lambda^\btheta_{1,n} = \Lambda^\btheta_{n}$ and $\bthetahat_{1,n} = \bthetahat_{n}$ for brevity.
By Wald's equation
\begin{align}\label{eq:wald}
    \KL\Ex_1^\btheta(T) &= \Ex_1^\btheta(\Lambda^\btheta_{T})\nonumber \\
    &=\Ex_1^\btheta(\widehat\Lambda_{T}) + \Ex_1^\btheta\left(\Lambda^\btheta_{T} - \widehat\Lambda_{T}\right).
\end{align}

The first term in \eqref{eq:wald} can be written as $b +\Ex_1^\btheta(\widehat\Lambda_{T}-b)$, where the overshoot $\Ex_1^\btheta(\widehat\Lambda_{T}-b)$ can be bounded by a constant independent of $b$ using a similar argument as was used in the proof of Theorem~\ref{thm:WL_UB} for bounding the overshoot.

For the second term in \eqref{eq:wald},
\begin{align}
    \Ex_1^\btheta(\Lambda^\btheta_{T} - \widehat\Lambda_{T})
    &= \Ex_1^\btheta\left(\sum_{n=1}^{T} \btheta^\top\bX_n - \frac{1}{2}\thetanorm - \bthetahat_{n}^\top \bX_n + \frac{1}{2}\norm{\bthetahat_n}\right) \\
    &= \Ex_1^\btheta\left(\sum_{n=1}^{T} \frac{1}{2}\thetanorm - \bthetahat_{n}^\top\btheta + \frac{1}{2}\norm{\bthetahat_n} \right), \\
    &=\frac{1}{2} \Ex_1^\btheta  \left(\sum_{n=1}^{T}\lVert\btheta- \bthetahat_{n}\rVert^2\right).
\end{align}
where the second equality follows by iterated expectation. From \cite{LORDEN_2005} we have the first order result 
\begin{equation}\label{eq:first_order}
    \Ex_1^\btheta(T) = \frac{b}{\KL}(1+o(1)).
\end{equation}
Using \eqref{eq:first_order} and modifying arguments from Lemma 16 of \cite{ROBBINS_1974} to a vector-valued $\btheta$ we can show that for $\eta_\delta = (1+\delta)b/\KL$ and any $\delta > 0,$ $\P_1^\btheta(T > \eta_\delta) \to 0$ sufficiently fast as $b \to \infty$ so that
\begin{equation}
   \Ex_1^\btheta  \left(\sum_{n=1}^{T}\lVert\btheta- \bthetahat_{n}\rVert^2\right) \leq \sum_{n=1}^{\eta_\delta}\Ex_1^\btheta\lVert\btheta- \bthetahat_{n}\rVert^2 + O(1).
\end{equation}

Applying \eqref{eq:MSE_condition}, we obtain
\begin{align}
\Ex_1^\btheta(\Lambda^\btheta_{T} - \widehat\Lambda_{T}) &=\frac{1}{2}\sum_{n=1}^{\eta_\delta}\Ex_1^\btheta\lVert\btheta- \bthetahat_{n}\rVert^2 + O(1) \\
    &= \frac{1}{2}\sum_{n=1}^{\eta_\delta}\MSE{\bthetahat_{n+1}} + O(1) \\
    &\leq\frac{1}{2} \sum_{n=2}^{\eta_\delta}\frac{\kappa(\btheta)}{n} + O(1) \\
    &= \frac{\kappa(\btheta)}{2} \log \eta_\delta + O(1) \\
    &= \frac{\kappa(\btheta)}{2}\log b + O(1). \label{eq: mse_sum_final}
\end{align}
Substituting \eqref{eq: mse_sum_final} into \eqref{eq:wald} and choosing $b = \log\gamma$ completes the proof.

\bibliographystyle{IEEEtran}
\bibliography{refs}
\end{document}